\documentclass[a4paper, reqno, 11pt]{amsart}
\usepackage{amsmath}
\usepackage{amsthm}
\usepackage{amssymb}
\usepackage{enumerate}
\usepackage{enumitem}
\usepackage{pifont}
\usepackage{mathrsfs}
\usepackage{adjustbox}
\usepackage[parfill]{parskip}
\usepackage{theoremref}

\DeclareMathOperator{\FG}{FG}
\DeclareMathOperator{\FC}{FC}

\theoremstyle{plain}
\newtheorem{thm}{Theorem}[section]
\newtheorem{cor}[thm]{Corollary}
\newtheorem{lem}[thm]{Lemma}
\newtheorem{prop}[thm]{Proposition}

\theoremstyle{definition}
\newtheorem{definition}[thm]{Definition}
\newtheorem{eg}[thm]{Example}
\newtheorem{rem}[thm]{Remark}

\newtheorem{op}[thm]{Open Problem}

\title{On separability properties in direct products of semigroups}
\author{Gerard O'Reilly, Martyn Quick, Nik Ru{\v s}kuc}
\thanks{Gerard O'Reilly would like to thank the School of Mathematics and Statistics of the University of St Andrews for supporting him in his graduate research.}
\address{School of Mathematics and Statistics, University of St Andrews, St Andrews, Scotland, UK}
\email{$\{$gao2, mq3, nr1$\}$@st-andrews.ac.uk}

\date{}

\begin{document}
	
\begin{abstract}
We investigate four finiteness conditions related to residual finiteness: complete separability, strong subsemigroup separability, weak subsemigroup separability and monogenic subsemigroup separability.
For each of these properties we examine under which conditions the property is preserved under direct products.
We also consider if any of the properties are inherited by the factors in a direct product.
We give necessary and sufficient conditions for finite semigroups to preserve the properties of strong subsemigroup separability and monogenic subsemigroup separability in a direct product.
\end{abstract}

\maketitle

\footnotetext{Data sharing not applicable to this article as no
  datasets were generated or analysed during the current study.}

\section{Introduction, Preliminaries and Summary of Results}

The purpose of this paper is to investigate the relationship between direct products of semigroups and their factors with respect to four separability properties.
We begin by defining separability properties in general and then specialize to the four of interest within this paper.

\begin{definition}
\thlabel{def:sep}
For a semigroup $S$ and a collection $\mathcal{C}$ of subsets of $S$, we say that $S$ has the \emph{separability property with respect to $\mathcal{C}$} if for any $s \in S$ and any $\mathcal{C}$-subset $Y \subseteq S \setminus \{s\}$ there exists a finite semigroup $P$ and homomorphism $\phi: S \to P$ such that $\phi(s) \notin \phi(Y)$.
In this case we say that $s$ can be \emph{separated} from $Y$ and that $\phi$ \emph{separates} $s$ from $Y$.
Equivalently, $S$ has the separability property with respect to $\mathcal{C}$ if for any $s \in S$ and any $\mathcal{C}$-subset $Y \subseteq S \setminus \{s\}$ there exists a finite index congruence $\rho$ on $S$ such that $[s]_{\rho} \neq [y]_{\rho}$ in~$S/\rho$ for all $ y \in Y$. 
In this case we say that $\rho$ \emph{separates} $s$ from $Y$.

For a semigroup $S$ we say that:
\begin{itemize}[nolistsep, noitemsep]
	\item $S$ is \emph{monogenic subsemigroup separable} (MSS) if $S$ has the separability property with respect to the collection of all monogenic subsemigroups;
	\item $S$ is \emph{weakly subsemigroup separable} (WSS) if $S$ has the separability property with respect to the collection of all finitely generated subsemigroups;
	\item$S$ is \emph{strongly subsemigroup separable} (SSS) if $S$ has the separability property with respect to the collection of all subsemigroups;
	\item $S$ is \emph{completely separable} (CS) if $S$ has the separability property with respect to the collection of all subsets.
\end{itemize}
\end{definition}

We note that separability properties are \emph{finiteness conditions.}
That is, if $\mathcal{P}$ is a separability property then any finite semigroup $S$ has $\mathcal{P}$.
This is because the identity map will separate any point $s$ from any subset of $S \setminus \{s\}$.

The best known separability property is \emph{residual finiteness}, which can be viewed as the separability property with respect to the collection of singleton subsets.
This is equivalent to having the separability property with respect to the collection of all finite subsets.
It was shown in \cite[Theorem 2]{Gray2009} that the direct product of two semigroups is residually finite if and only the factors are residually finite.
The fact that residual finiteness of the factors implies the residual finiteness of the direct product is universally true for algebraic structures.
However, it is non-trivial to show that if a direct product of semigroups is residually finite then both factors are residually finite.
Our aim is to determine for which of our four separability properties we can obtain an analogous results.
In the cases where we fail to obtain such a result, we investigate further to determine what can be said about the separability properties of direct products.

By replacing the word ``semigroup" with ``group" and replacing ``subsemigroup" with ``subgroup", we can define analogous separability properties in the class of groups.
The properties have been widely studied in groups, under various different names.
\emph{Monogenic subgroup separable} groups are also known as \emph{cyclic subgroup separable} groups or $\Pi_C$ groups.
In \cite[Theorem 4]{Stebe}, Stebe showed that the direct product of two monogenic subgroup separable groups is itself monogenic subgroup separable.
\emph{Weakly subgroup separable} groups are also known simply as \emph{subgroup separable} groups or as \emph{locally extended residually finite} (LERF) groups.
It is not true that the direct product of two weakly subgroup separable groups is necessarily weakly subgroup separable.
An example is given by Allenby and Gregorac in \cite[``The Example"]{Allenby}.
\emph{Strongly subgroup separable groups} are also known as \emph{extended residually finite} (ERF) groups.
The direct product of two strongly subgroup separable groups is known to be strongly subgroup separable.
In \cite{Allenby} Allenby and Gregorac attribute this result to Mal'cev in \cite{Malcev}.
It was shown in \cite[Lemma 2.4]{me} that a group is \emph{completely separable} if and only if it is finite.  
Hence, trivially, the direct product of two CS groups is itself CS\@.
The results for groups are recorded in Table \ref{Grouptable}.

\begin{table}
\begin{adjustbox}{max width=\textwidth}
\begin{tabular}{c | c c c c}
		\textbf{Group Property $\mathcal{P}$}  & \textbf{CS} & \textbf{SSS} & \textbf{WSS} & \textbf{MSS} \\
		\hline
		$G$, $H$ are $\mathcal{P}$ & \ding{51} & \ding{51} & \ding{55} & \ding{51} \\
		$\implies G \times H$ is $\mathcal{P}$ & \cite[Lem 2.4]{me} & \cite{Malcev} & \cite[``The Example"]{Allenby} & \cite[Thm 4]{Stebe}\\
\end{tabular}
\end{adjustbox}
\vspace{1em}
\caption{Separability properties of direct products of groups}
\label{Grouptable}
\end{table}

For groups $G$ and $H$, both $G$ and $H$ are always isomorphic to subgroups of $G \times H$.
As subgroups inherit all four of these properties, an easy extension of \cite[Proposition 2.3]{me}, if $G \times H$ has one of these properties then so will both $G$ and $H$.
However, the situation for algebras in general, and for semigroups in particular, may not always be so straightforward.
Indeed, in \cite{de Witt} de Witt was able to show that there exist monounary algebras $A$ and $B$ such that $A$ is not residually finite but $A \times B$ is completely separable.
So factors of a direct product of monounary algebras which has separability property $\mathcal{P}$ need not be $\mathcal{P}$ themselves.
We shall observe that within the class of semigroups a similar situation occurs with the properties of weak subsemigroup separability and monogenic subsemigroup separability.

Before we begin, we give a taste of how separability properties can interact with the semigroup direct product by considering the situation where one of the factors is $\mathbb{N}$.

\begin{thm}
\thlabel{Ncharac}
For a semigroup $S$, the following fully characterise the separability properties of $\mathbb{N} \times S$:
\begin{enumerate}[nolistsep, noitemsep]
	\item $\mathbb{N} \times S$ is completely separable if and only if $S$ is completely separable;
	\item $\mathbb{N} \times S$ is weakly subsemigroup separable but not strongly subsemigroup separable if and only if $S$ is residually finite but not completely separable;
	\item $\mathbb{N} \times S$ is not residually finite if and only if $S$ is not residually finite.
\end{enumerate}
\end{thm}

\begin{proof}
(1) As $\mathbb{N}$ is CS (\cite[Corollary 4]{Golubov70}), the result follows from \thref{comp sep}, which states that the direct product of two semigroups is CS if and only if both factors are CS\@.

(2) $(\Rightarrow)$ As WSS semigroups are residually finite (\cite[Proposition 2.1]{me}), it follows that $S$ is residually finite by \cite[Theorem 2]{Gray2009}.
The fact $S$ cannot be CS is a consequence of \thref{comp sep}. 

$(\Leftarrow)$ The fact that $\mathbb{N} \times S$ is WSS follows from \cite[Proposition 5.1]{me}, which states that a residually finite semigroup which has $\mathbb{N}$ as homomorphic image is WSS\@.
The fact that $S$ cannot be SSS follows from \thref{strong no}, which establishes a more general result.

(3) This follows from \cite[Theorem 2]{Gray2009}.
\end{proof}

\begin{rem}
In the statement of \thref{Ncharac}, the semigroup $\mathbb{N}$ can be replaced by any CS semigroup which has $\mathbb{N}$ as a homomorphic image and the result still holds.
Such semigroups include all free semigroups and free commutative semigroups.
\end{rem}

The characterisation given in \thref{Ncharac} highlights the intriguing interaction of semigroup separability properties with the direct product.
On one hand, it tells us that we can take the direct product of $\mathbb{N}$ with a SSS semigroup and the resulting semigroup need not be SSS, even though both factors are (and in fact $\mathbb{N}$ is CS)\@.
On the other hand we can take a semigroup which is not WSS but its direct product with $\mathbb{N}$ is WSS\@.
Consequently there is no guarantee that the direct product preserves separability properties nor that the factors of a direct product inherit separability properties.

By working through the separability properties in turn, we are able to say which of them are preserved under direct products.
We also consider which separability properties are inherited from the direct product by its factors.
These results are recorded in Table \ref{Semigrouptable}.

\begin{table}
	\begin{adjustbox}{max width=\textwidth}
		\begin{tabular}{c | c c c c}
			\textbf{Semigroup Property $\mathcal{P}$} & \textbf{CS} & \textbf{SSS} & \textbf{WSS} & \textbf{MSS} \\
			\hline
			$S$, $T$ are $\mathcal{P}$ & \ding{51} & \ding{55} & \ding{55} & \ding{55} \\
			$\implies S \times T$ is $\mathcal{P}$ & Thm \ref{comp sep} & Ex \ref{strong k-nilpotent} & Ex \ref{left-zero counter} & Ex \ref{monogenic eg} \\
			\hline
			$S \times T$ is $\mathcal{P}$ & \ding{51} & \ding{51} & \ding{55} & \ding{55} \\
			$\implies S, T$ are $\mathcal{P}$& Thm \ref{comp sep} & Thm \ref{strong yes} & Ex \ref{N cross Z} & Ex \ref{N cross Z}
		\end{tabular}
	\end{adjustbox}
\vspace{1em}
\caption{Separability properties of direct products of semigroups}
\label{Semigrouptable}
\end{table}

It turns out that the properties of strong subsemigroup separability, weak subsemigroup separability and monogenic subsemigroup separability are not necessarily preserved in the direct product.
This motivates the following definition.

\begin{definition}
Let $\mathcal{P}$ be one of the following properties: strong subsemigroup separability, weak subsemigroup separability or monogenic subsemigroup separability. 
We say that a semigroup $T$ is \emph{$\mathcal{P}$-preserving} (in direct products) if for every semigroup $S$ which has property $\mathcal{P}$, the direct product $S \times T$ also has property $\mathcal{P}$.	
\end{definition}

We note that if a semigroup $T$ is $\mathcal{P}$-preserving then $T$ must have property $\mathcal{P}$.
This is because the trivial semigroup has property $\mathcal{P}$ and $T$ is isomorphic to the direct product of itself with the trivial semigroup.
In this paper we focus on characterising when finite semigroups are $\mathcal{P}$-preserving.

We now recall some notions of semigroups and direct products that we shall use throughout the paper.
For a congruence $\rho$ on a semigroup $S$, we denote the $\rho$-class of $s \in S$ by $[s]_{\rho}$.	
For an ideal $I \leq S$, we can define a congruence $\rho_I$ on $S$ by:
\[s \, \rho_I \, t \text{ if and only if } s, t \in I \text{ or } s=t.\]
This congruence is known as the \emph{Rees congruence} on $S$ by $I$.
We denote the quotient of this congruence by $S/I$ and the congruence class of $s \in S$ by $[s]_I$.
Note that if $s \in S \setminus I$, then $[s]_I=\{s\}$.

Green's relation $\mathcal{J}$ on a semigroup $S$ is an equivalence relation given by
\[s \, \mathcal{J} \, t \text{ if and only if } S^1sS^1=S^1tS^1,\]
where $S^1$ is the semigroup $S$ with an identity adjoined.
That is, two elements are $\mathcal{J}$-related if and only if they generate the same two-sided ideal.
Green's relation $\mathcal{L}$ on a semigroup $S$ is an equivalence relation given by
\[s \, \mathcal{L} \, t \text{ if and only if } S^1s=S^1t.\]
That is, two elements are $\mathcal{L}$-related if and only if they generate the same left ideal.
Similarly, Green's relation $\mathcal{R}$ on a semigroup $S$ is an equivalence relation given by
\[s \, \mathcal{R} \, t \text{ if and only if } sS^1=tS^1.\]
That is, two elements are $\mathcal{R}$-related if and only if they generate the same right ideal.
Green's relation $\mathcal{H}$ on a semigroup is an equivalence relation on $S$ given by
\[s \, \mathcal{H} \, t \text{ if and only if } s \, \mathcal{L} \, t  \text{ and } s \, \mathcal{R} \, t .\]
An $\mathcal{H}$-class $H$ is a subgroup of $S$ if and only if $H^2=H$.

For semigroups $S$ and $T$, the projection map $\pi_S : S \times T \to S$ is given by $(s, t) \mapsto s$.
Similarly, $\pi_T: S \times T \to T$ is given by $(s, t) \mapsto t$.
Projection maps are homomorphisms.
Throughout $\mathbb{N}$ denotes the set $\{1,2,3, \dots\}$.

This paper is organised by property.
In Section 2 we show that the direct product of two semigroups is CS if and only if both semigroups are CS (\thref{comp sep}).
In Section 3 we show that if the direct product of two semigroups is SSS then both factors must be SSS (\thref{strong yes}).
We also give a characterisation of when a finite semigroup is SSS-preserving (\thref{finite strong}).
In Section 4 we show that weak subsemigroup separability need not be preserved under direct products nor inherited by factors of a direct product.
Whilst a classification of finite WSS-preserving semigroups is not known, we give some classes of finite semigroups that are WSS-preserving.
In Section 5 we show that monogenic subsemigroup separability need not be preserved under direct products nor inherited by factors of a direct product.
In this case we are able to classify when a finite semigroup is MSS-preserving (\thref{monogenic finite}).

\section{Complete Separability}

Complete separability proves to be the best behaved of our separability properties with respect to the direct product.
In order to show this, we will use a characterisation of complete separability given by Golubov in \cite{Golubov70}.
This characterisation depends upon the following definition.

\begin{definition}
	For a semigroup $S$ and $a, b \in S$ define 
	\[[a:b]=\{(u, v) \in S^1 \times S^1 \mid ubv=a\}.\]	
\end{definition}

Using this definition, we have the following characterisation of CS semigroups.

\begin{thm} \thlabel{Golubov} \normalfont{\cite[Theorem 1]{Golubov70}}
	\textit{A semigroup $S$ is completely separable if and only if for each $a \in S$ the set $\{[a:s] \mid s \in S\}$ is finite.}
\end{thm}

Using this, we give the following result.

\begin{thm}
	\thlabel{comp sep}
	The direct product of semigroups $S$ and $T$ is completely separable if and only if $S$ and $T$ are completely separable.
\end{thm}

\begin{proof}
	$(\Leftarrow)$ The fact that the direct product of two CS semigroups is CS was shown by Golubov in \cite[Lemma 2]{Golubov72}.
	Here we present an independent proof.
	Assume that $S$ and $T$ are CS\@.
	Let $(s, t) \in S \times T$.
	There exists a finite semigroup $U$ and homomorphism $\phi: S \to U$ such that $\phi(s) \notin \phi(S \! \setminus \! \{s\})$.
	There exists a finite semigroup $V$ and homomorphism $\psi: T \to V$ such that $\psi(t) \notin \psi(T \! \setminus \! \{t\})$. 
	Then $\phi \times \psi: S \times T \to U \times V$ given by $(\phi \times \psi)(a, b)=(\phi(a), \psi(b))$ is a homomorphism that separates $(s, t)$ from its complement.
	
	$(\Rightarrow)$ We will prove the contrapositive, so assume that $S$ is not CS\@.
	We show that $S \times T$ is not CS\@.
	By \thref{Golubov} there exists some $a \in S$ such that the set $\{[a:s] \mid s \in S \}$ is infinite.
	Let $\sim$ be a finite index congruence on $S \times T$.
	Fix $t \in T$.
	Then there exist $b, c \in S$ such that $[a:b] \neq [a:c]$ and $(b, t) \sim (c, t)$.
	Assume, without loss of generality, that there exist $u, v \in S^1$ such that $ubv=a$ but $ucv \neq a$.
	We split into two cases.
	The first case is that $u, v \in S$.
	Then 
	\[(a, t^3)=(u, t)(b, t)(v, t) \sim (u, t)(c, t)(v, t)=(ucv, t^3).\]
	The second case is that either $u=1$ or $v=1$.
	We will deal with the case that $u=1$.
	A similar argument deals with the case that $v$ is not in $S$.
	Then
	\[(a, t^3)=(b,t)(v, t^2) \sim (c, t)(v, t^2)=(cv, t^3)=(ucv, t^3).\]
	In either case we cannot separate $(a, t^3)$ from its complement. 
	Therefore $S \times T$ is not CS\@.
\end{proof}

\section{Strong Subsemigroup Separability}

Strong subsemigroup separability is not as well behaved as complete separability or residual finiteness with respect to the direct product.
Indeed, the direct product of two SSS semigroups need not be itself SSS\@.
Golubov showed this in \cite{Golubov72} with the direct product of two infinite SSS semigroups.
We show that even when one of the factors is finite, the direct product of two SSS semigroups need not be SSS (\thref{strong k-nilpotent}).
However, it is true that if a direct product of semigroups is SSS then both factors are SSS (\thref{strong yes}).
We shall begin by stating a lemma which provides a sufficient condition for the direct product of two semigroups to fail to be SSS\@.

First, we recall some facts about monogenic semigroups.
A monogenic semigroup $T=\langle x \rangle $ is either isomorphic to $\mathbb{N}$ or it is finite.
In the case that $T$ is finite, there exist positive integers $m$ and $r$ such that $x^m=x^{m+r}$.
The smallest such value of $m$ is called the \emph{index} of $T$ and the smallest such value of $r$ is called the \emph{period} of $T$.
For more on monogenic semigroups see \cite[Theorem 1.2.2]{Howie95}.
To simplify statements, we will say that the infinite monogenic semigroup has index $\infty$.

\begin{lem}
\thlabel{strong no}
	Let $S$ be a semigroup which is not completely separable and let $T$ be a monogenic semigroup with index $m \geq 4$.
	Then $S \times T$ is not strongly subsemigroup separable.
\end{lem}

\begin{proof}
	Let $x$ be the generator of $T$.
	As $S$ is not CS, by \thref{Golubov} there exists some $a \in S$ such that the set $\{[a:s] \mid s \in S\}$ is infinite. 
	Let $\sim$ be a finite index congruence on $S \times T$.
	Then there exist $b,  c \in S$ such that $[a:b] \neq [a:c]$ and $(b, x) \sim (c, x)$.
	Assume, without loss of generality, that there exist $u,  v \in S^1$ such that $ubv=a$ but $ucv \neq a$.
	Let $U=\langle (S \setminus \{a\}) \times \{x^3\} \rangle \leq S \times T $. 
	Since $m\geq 4$ we have $x^3 \notin \{ x^{3k} | k>1\}$, and hence $(a,x^3)\not\in U$.
	
	We split into two cases.
	The first is that $u, v \in S$.
	Then
	\[(a, x^3)=(u,x)(b, x)(v, x) \sim (u, x)(c, x)(v, x)=(ucv, x^3) \in U.\]
	The second case is when either $u=1$ or $v=1$.
	We will deal with the case that $u=1$.
	Then
	\[(a, x^3)=(b, x)(v, x^2) \sim (c, x)(v, x^2)=(ucv, x^3) \in U.\]
	In either case we cannot separate $(a, x^3)$ from $U$.
	A similar argument deals with the case that $v$ is not in $S$.
	Therefore $S \times T$ is not SSS\@.
\end{proof}

As there exist SSS semigroups which are not CS, for example \cite[Examples 5.4 and 5.11]{me}, we can use \thref{strong no} to construct an example of a direct product of a SSS semigroup with a finite semigroup which itself is not SSS\@.

\begin{eg}
	\thlabel{strong k-nilpotent}
	As an immediate consequence of \thref{strong no}, if we take the direct product of an SSS semigroup $S$ which is not CS with the finite 4-nilpotent semigroup $\langle x \mid x^4=x^5 \rangle$, the resulting semigroup is not SSS\@.
\end{eg}

Despite the negative nature of \thref{strong no}, it actually plays a key part in the proof of the following result.

\begin{thm}
	\thlabel{strong yes}
	If $S \times T$ is strongly subsemigroup separable then both $S$ and $T$ are strongly subsemigroup separable.
\end{thm} 

\begin{proof}
	Assume $S \times T$ is SSS\@.
	Without loss of generality we show that $S$ is SSS\@.
	We split into two cases. 
	
	The first case is that $T$ contains an idempotent. 
	Then $S$ is isomorphic to a subsemigroup of the SSS semigroup $S \times T$ and hence $S$ must also be SSS by \cite[Proposition 2.3]{me}.
	
	The second case is that $T$ has no idempotents.
	For a contradiction assume that $S$ is not SSS\@.
	Then $S$ is certainly not CS\@.
	As $T$ is idempotent free, it contains a copy of $\mathbb{N}$. 
	Then $S \times \mathbb{N} \leq S \times T$. 
	But $S \times \mathbb{N}$ is not SSS by \thref{strong no}.
	This contradicts the strong subsemigroup separability of $S \times T$ and therefore it must be the case that $S$ is SSS\@.
\end{proof}

The rest of this section is dedicated to establishing when a finite semigroup is SSS-preserving.
To do this we show that SSS semigroups enjoy a separability property, in the sense of \thref{def:sep}, with respect to a more general class of subsets than just subsemigroups.

\begin{prop}
	\thlabel{V^{n+1}}
	Let $S$ be a strongly subsemigroup separable semigroup and let $V \subseteq S$ such that $V^{n+1} \subseteq V$ for some $n \in \mathbb{N}$ and let $ s \in S \setminus V$.
	Then $s$ can be separated from $V$.	
\end{prop}

\begin{proof}
	We proceed by induction on $n$.
	The base case $n=1$ corresponds to $V$ being a subsemigroup, and the result follows from $S$ being SSS\@.
	
	Now consider $n>1$.
	Our inductive hypothesis is that for all $U \subseteq S$ such that $U^{k+1} \subseteq U$ for some $k \in \{1, 2, \dots, n-1\}$, and for all $s' \in S\setminus U$, we can separate $s'$ from $U$.
	Let $V$ and $s$ be as in the statement of the proposition.
	If $s \notin \langle V \rangle$ then as $S$ is SSS, we can separate $s$ from $\langle V \rangle $ and in particular from $V$.
	
	So suppose that $s \in \langle V \rangle$.
	Note that as $V^{n+1} \subseteq V$, we have that $\langle V \rangle =V \cup V^2 \cup \dots \cup V^n$ and that $V^n$ is a subsemigroup of $S$.
	Let $L=\{\ell \in V \mid s \in \ell V^i, 1 \leq i \leq n-1\}$, which is some set of left-divisors of $s$.
	As $s \in \langle V \rangle$, we have that $L$ is non-empty.
	Let $Z =V \setminus L$.
	Then we claim that $s \notin \langle Z \rangle$.
	Indeed if $s \in \langle Z \rangle$, then $s=zw$ for some $z \in Z$ and $w \in Z^k$, where $k \in \mathbb{N} \cup \{0\}$.
	If $k =0$ then $s =z \in Z \subseteq V$, which is a contradiction.
	Otherwise, as $V^{n+1} \subseteq V$, we have $w \in V^i$ for some $i \in \{1, 2, \dots, n\}$.
	If $w \in V^n$, then $s \in V^{n+1} \subseteq V$, which is a contradiction.
	Then $w \in V^i$ for some $1 \leq i \leq n-1$.
	But then $z \in L$, which is a contradiction.
	So $s \notin \langle Z \rangle$.
	As $S$ is SSS, $s$ can be separated from $\langle Z \rangle$ and in particular $s$ can be separated from $Z$.
	
	For $i \in \{1, 2, \dots, n-1\}$, define $X_i=V \cap V^{i+1}$.
	Then we claim that
	\begin{equation}
	\label{eq1}
	X_i^{n-i+1} \subseteq X_i.
	\end{equation}
	To see this let $x_1, x_2, \dots, x_{n-i+1} \in X_i$.
	Firstly, as $X_i = V \cap V^{i+1}$, we have $x_1 \in V^{i+1}$ and $x_2, \dots, x_{n-i+1} \in V$.
	Then 
	\[x_1x_2 \dots x_{n-i+1} \in V^{(i+1)+(n-i+1-1)}=V^{n+1} \subseteq  V.\]
	Secondly, noting that $n-i+1 \geq 2$, we have $x_1, x_2 \in V^{i+1}$.
	Then 
	\[x_1x_2 \dots x_{n-i+1} \in V^{2(i+1)+(n-i+1-2)}=V^{n+i+1} \subseteq V^{i+1}.\]
	Hence $x_1x_2 \cdots x_{n-i+1} \in V \cap V^{i+1}=X_i$ and we conclude $X_i^{n-i+1} \subseteq X_i$ and (1) holds.
	As $s \notin V$ we have that $s \notin X_i$.
	
	Since $1 \leq n-i \leq n-1$ and $X_i^{n-i+1} \subseteq X_i$, by our inductive hypothesis we can separate $s$ from each $X_i$.
	For $i \in \{1, \dots, n-1\}$, define $L_i=\{\ell \in L \mid s \in \ell V^i\}$.
	Note that $L=\bigcup_{1 \leq i \leq n-1}{L_i}$.
	We show that $s$ can be separated from each $L_i$.
	Suppose that $s$ cannot be separated from some $L_i$.
	We claim that $sV^{n-1} \subseteq V^{i} \cap V^n$.
	First for a contradiction, assume that there exists $u \in V^{n-1}$ such that $su \in S \setminus V^n$.
	Then as $V^n$ is a subsemigroup and $S$ is SSS it must be the case that $su$ can be separated from $V^n$.
	Let $\sim$ be a finite index congruence which separates $su$ from $V^n$.
	As $s$ cannot be separated from $L_i$, there exists $v \in L_i \subseteq V$ such that $s \sim v$.
	But then $su \sim vu \in V^n$ which is a contradiction.
	Hence if $s$ cannot be separated from $L_i$, we have $sV^{n-1} \subseteq V^n$.
	As $L_i$ is non-empty, we have that $s \in V^{i+1}$.
	But as $s \in V^{i+1}$ and $V^{n+1} \subseteq V$, we have that $sV^{n-1} \subseteq V^i$.
	Hence $sV^{n-1} \subseteq V^n \cap V^i$ as claimed.
	From this we get 
	\begin{equation}
	\label{eq2}
	sV^{n} \subseteq V \cap V^{i+1}=X_i.
	\end{equation}
	
	Now let $\sim$ be a finite index congruence that separates $s$ from $X_i$.
	As $s$ cannot be separated from $L_i$, there exists $\ell \in L_i$ such that $\ell \sim s$.
	But as $\ell \in L_i$, there exists $u \in V^i$ such that $\ell u=s$.
	Then, by the compatibility of $\sim$ we have 
	\[s=\ell u\sim su \sim su^2 \sim \dots \sim su^{n-1} \sim su^n.\]
	By the transitivity of $\sim$ we obtain that $s \sim su^n$.
	As $u \in V^{i}$, we have that $u^n \in V^{in}$.
	But as $V^{n+1} \subseteq V$, we conclude that $u^n \in V^{n}$.
	So by Equation (\ref{eq2}) we have $su^n\in sV^n\subseteq X_i$.
	That is $\sim$ does not separate $s$ from $X_i$.
	This is a contradiction and so $s$ can be separated from $L_i$.
	
	As $V=Z \cup L_1 \cup L_2 \cup \dots \cup L_{n-1}$, and for each of the sets in this finite union there exists a finite index congruence which separates $s$ from said set, the intersection of these congruences is a finite index congruence which separates $s$ from $V$.
\end{proof}

The fact that SSS semigroups satisfy the additional separability property of \thref{V^{n+1}} allows us to characterise finite semigroups which are SSS-preserving.
Our characterisation relies on the notion of indecomposability of elements.
For a semigroup $S$ and an element $s \in S$, we say that $s$ is \emph{decomposable} if $s \in S^2$.
In this case, there exist $t, u \in S$ such that $s=tu$.
Otherwise we say that $s$ is \emph{indecomposable.}

\begin{thm}
\thlabel{finite strong}
	A finite semigroup $P$ is strong subsemigroup separability preserving if and only if every element of $P$ is indecomposable or belongs to a subgroup.
\end{thm}

\begin{proof}
	($\Rightarrow$) We prove the contrapositive so assume that $p \in P$ is not contained in a subgroup but there exist $s, t \in P$ such that $st=p$.
	As $p$ is not contained in a subgroup, we have that $p^n\neq p$ for all $ n\geq 2$.
	Let $G$ be an infinite SSS group (e.\@g.\@ \cite[Example 5.4]{me}).
	Let $U \leq G \times P$ be generated by the set $\{(g, p) \mid g \in G \setminus \{1_G\}\}$.
	As $p^n \neq p$ for all $n \geq 2$, we have that $(1_G, p) \notin U$.
	Let $\sim$ be a finite index congruence on $G \times P$.
	Then there exist $g, h \in G$ with $g\neq h$ such that $(g, s) \sim (h, s)$.
	Then 
	\[(1_G, p)=(g, s)(g^{-1}, t) \sim (h, s)(g^{-1}, t)=(hg^{-1}, p) \in U.\]
	Hence $G \times P$ is not SSS and therefore $P$ is not SSS-preserving. 
	
	($\Leftarrow$) Now assume that $P$ is a finite semigroup in which every element not contained in a subgroup is indecomposable.
	Let $S$ be an SSS semigroup, let $U \leq S \times P$ and let $(s, p) \in (S \times P) \setminus U$. 
	If $ s \notin \pi_S(U) \leq S$, then we can separate $(s, p)$ from $U$ by factoring through $S$ and invoking the strong subsemigroup separability of $S$.
	If $ p \notin \pi_P(U)$, then $\pi_P$ separates $(s, p)$ from $U$.
	
	Now assume that both $s \in \pi_S(U)$ and $p \in \pi_P(U)$.
	If $p$ is not contained within a subgroup of $P$, then $p$ is indecomposable.
	Then $(s, p)$ is indecomposable in $S \times P$.
	Let $I=(S \times U) \setminus \{(s, p)\}$.
	Then $I$ is an ideal of finite complement in $S \times U$ and $[(s, p)]_I=\{(s, p)\}$.
	In particular, $(s, p)$ is separated from $U$ in the Rees quotient of $S \times P$ by $I$.
	
	The final case to consider is that $p$ is contained in a subgroup of $P$.
	Then for some $n \in \mathbb{N}$, we have that $p^{n+1}=p$.
	Let $V =\pi_S(U \cap (S \times \{p\})$.
	First note that $s \notin V$ as $(s, p) \notin U$.
	Secondly, as $p^{n+1}=p$, we have $V^{n+1} \subseteq V$.
	Then by \thref{V^{n+1}} we have that $s$ can be separated from $V$.
	So there exists a finite semigroup $Q$ and a homomorphism $\phi: S \to Q$ such that $\phi(v) \notin \phi(V)$.
	Define $\overline{\phi} : S \times P \to Q \times P$ by $(a, b) \mapsto (\phi(a), b)$.
	Then $\overline{\phi}$ is a homomorphism which separates $(s, p)$ from $U$.
	Hence $S \times P$ is SSS and so $P$ is SSS-preserving.
\end{proof}

\begin{rem}
	As the set $S^2$ is an ideal of a semigroup $S$, \thref{finite strong} is equivalent to saying that a finite semigroup $P$ is SSS-preserving if and only if $P$ is the ideal extension of a union of groups by a null semigroup. 
\end{rem}

\begin{cor}
The following families of finite semigroups are strong subsemigroup separability preserving: groups, Clifford semigroups, completely simple semigroups, completely regular semigroups, bands, and null semigroups.
\end{cor}

\begin{cor}
A finite monoid is strong subsemigroup separability preserving if and only if it is a union of groups.
\end{cor}

We conclude this section with some open problems.

\begin{op}
Is there a characterisation of when the direct product of two strongly subsemigroup separable semigroups is itself strongly subsemigroup separable?
\end{op}

\begin{op}
\thlabel{2strong}
Is it true that the direct product of two strongly subsemigroup separable semigroups is weakly subsemigroup separable?
\end{op}

\section{Weak Subsemigroup Separability}

The behaviour of weak subsemigroup separability with respect to the direct product is more complicated than that of the separability properties we have already discussed. 
This is exhibited in \thref{2strong}, which states that even when we strengthen the factor semigroups to be SSS, we still do not know if this guarantees that the direct product is WSS\@.
We also see for the first time an example of a separability property which is not necessarily inherited by the factors of a direct product (\thref{N cross Z}).
This complexity means the results in this section are less comprehensive than those for the other separability properties considered.
We show that the direct product of a WSS semigroup with a finite semigroup is not necessarily WSS (\thref{left-zero counter}).
Although we do not characterise when a finite semigroup is WSS-preserving, we show that finite nilpotent semigroups are WSS-preserving (\thref{weak k-nilpotent}).
We are able to show that the direct product of two WSS semigroups is MSS (\thref{2WSS}).
We begin with the example that shows that weak subsemigroup separability may not be inherited by the factors of a direct product.

\begin{eg}
\thlabel{N cross Z}
Consider the semigroup $\mathbb{N} \times \mathbb{Z}$.
As this is a residually finite semigroup with $\mathbb{N}$ as a homomorphic image, it is WSS by \cite[Proposition 5.1]{me}.
However, $\mathbb{Z}$ is not WSS by \cite[Example 2.6]{me}.
\end{eg}

\begin{rem}
\thlabel{MSSfactor}
While \cite[Example 2.6]{me} states that $\mathbb{Z}$ is not WSS, the subsemigroup used in demonstrating this is $\mathbb{N}$.
As $\mathbb{N}$ is monogenic, we have that $\mathbb{Z}$ is not MSS\@.
Therefore $\mathbb{N} \times \mathbb{Z}$ is also an example of a direct product which is MSS but one of the factors is not MSS\@.
\end{rem}

It also turns out that the direct product of two WSS semigroups need not be WSS, even when one of the factors in finite.
To show this we first need to construct a specific instance of a WSS semigroup.

\begin{definition}
Let $\FC_2$ be the free commutative semigroup on the set $\{a, b\}$.
Let $N=\{x_z \mid z \in \mathbb{Z}\} \cup \{0\}$ be a null semigroup with zero element $0$.
Let $\phi: \FC_2 \to \mathbb{Z}$ be given by $a^ib^j \mapsto i-j$.
Consider the set $\mathcal{S}[\FC_2, \mathbb{Z}, \phi]=\FC_2 \cup N$. 
We extend the multiplication on $\FC_2$ and $N$ in the following manner.
For $w \in \FC_2$ and $x_z \in N$, define
\begin{align*}
&x_z \cdot w=x_{z +\phi(w)}, \\
&w \cdot x_z=x_{z-\phi(w)}, \\
&0 \cdot w=w \cdot 0 =0.
\end{align*}
Then $\mathcal{S}[\FC_2, \mathbb{Z}, \phi]$ is an example of semigroup of type given in \cite[Construction 5.5]{me} and is WSS by \cite[Proposition 5.8]{me}.
\end{definition}

\begin{eg}
	\thlabel{left-zero counter}
	Let $S=\mathcal{S}[\FC_2, \mathbb{Z}, \phi]$ and let $L$ be a non-trivial left-zero semigroup.
	Then $S \times L$ is not WSS\@.
	
	Let $y, z \in L$ be distinct. Let 
	\[U=\langle(a,y), (x_1, z)\rangle \leq S \times L.\]
	Then one computes that
	\[U=\{(a^i, y) \mid i \in \mathbb{N}\} \cup \{(x_i, z) \mid i \in \mathbb{N}\} \cup \{(x_i, y) \mid i \in \mathbb{Z}\} \cup \{(0, y), (0, z)\}.\]
	In particular, $(x_0, z) \notin U$. 
	Let $\sim$ be a finite index congruence on $S \times L$. 
	Then there exist $i, j \in \mathbb{N}$ with $i < j$ such that $(x_i, z) \sim (x_j, z)$. 
	Then
	\[(x_0, z)=(x_i, z)(b^i, z) \sim (x_j, z)(b^i, z)=(x_{j-i}, z) \in U.\]
	Hence $S \times L$ is not weakly subsemigroup separable.
	By a similar argument, it can be shown that the direct product of $S=\mathcal{S}[\FC_2, \mathbb{Z}, \phi]$ with a non-trivial right-zero semigroup is not WSS\@.
\end{eg}

\thref{left-zero counter} shows that (non-trivial) finite left-zero and right-zero semigroups are not weakly subsemigroup separability preserving.
This is somewhat surprising as finite left-zero and right-zero semigroups are strongly subsemigroup preserving by \thref{finite strong} and also turn out to be monogenic subsemigroup preserving by \thref{monogenic finite}.
Adding yet another twist to the story, the following theorem shows that finite nilpotent semigroups are WSS-preserving, even though this class contains semigroups which are neither SSS-preserving nor MSS-preserving. 
For example the semigroup with presentation $\langle \, x \mid x^3=x^4 \,\rangle$ is a finite 3-nilpotent semigroup which fails to meet the criteria to be either SSS-preserving or MSS-preserving.

\begin{thm}
	\thlabel{weak k-nilpotent}
	The direct product of a weakly subsemigroup separable semigroup with a residually finite nilpotent semigroup is weakly subsemigroup separable.
\end{thm}

\begin{proof}
	Let $S$ be a WSS semigroup and let $N$ be a residually finite $k$-nilpotent semigroup with zero element 0, for some $k \in \mathbb{N}$. 
	Let $U \leq S \times N$ be finitely generated and let $(s,n) \in (S \times N) \! \setminus \! U$. 
	Fix a finite generating set for $U$, which has the form
	\[(X_0 \times \{0\}) \cup (X_1 \times \{n_1\}) \cup \cdots \cup (X_j \times \{n_j\}),\]
	where $n_1, n_2, \dots, n_j \in N \setminus\{0\}$ and $X_0, X_1, \dots, X_j$ are finite subsets of $S$.
	Let $X= \bigcup_{i=0}^j{X_i}$ and let $T=\langle X \rangle \leq S$. 
	
	Let $Z=\pi_N(U)=\{z_0, z_1, z_2, \cdots, z_m\}$ where $z_0=0$.
	Note $Z$ is finite as $k$-nilpotent semigroups are locally finite.
	For $0 \leq i \leq m$ let 
	\[Y_i=\pi_S(U \cap (S \times \{z_i\})).\]
	Then for $1 \leq i \leq m$ the set $Y_i$ is finite. 
	To see this first note that $Y_i$ is a subset of $T$.
	Then for $y \in Y_i$, we can write $y$ as a product of elements of $X$.
	The maximum length of the product is $k-1$, as $N$ is a $k$-nilpotent semigroup and $z_i$ is a non-zero element of $N$. 
	As $X$ is a finite set, there are only finitely many such $y$ and so $Y_i$ is finite.
	
	Now consider $Y_0$.
	Certainly $T^k \subseteq Y_0$ as $N$ is a $k$-nilpotent semigroup.
	We have that $Y_0 \setminus T^k$ is finite as any element of $Y_0 \setminus T^k$ can be expressed as product over $X$ of length at most $k-1$.
	
	If $s \notin \pi_S(U)=T$ then we can separate $(s, n)$ from $U$ by factoring through $S$ and invoking the weak subsemigroup separability of $S$. 
	Similarly, if $t \notin \pi_N(U)=Z$ then we can separate $(s,n)$ from $U$ by factoring through $N$ and using the residual finiteness of $N$.
	
	Now assume that $s \in \pi_S(U)$ and $n \in \pi_N(U)$.
	Then either
	\begin{enumerate}[noitemsep, label=(\roman*)]
		\item $n=0$ and $s \notin Y_0$, or 
		\item $n=z_i$ for some $1 \leq i \leq m$ and $s \notin Y_i$. 
	\end{enumerate}
	
	(i) First note that $T^k \leq S$ is finitely generated by the set
	\[X^k \cup X^{k+1} \cup \dots \cup X^{2k-1}.\]
	As $s \notin T^k$ and $S$ is WSS, there exists a finite semigroup $P_1$ and homomorphism $\phi_1: S \to P_1$ such that $\phi_1(s) \notin \phi_1(T^k)$. 
	As $s \notin Y_0 \setminus T^k$ and $S$ is residually finite, there exists a finite semigroup $P_2$ and homomorphism $\phi_2: S \to P_2$ such that $\phi_2(s) \notin \phi_2(Y_0 \setminus T^k)$. 
	Therefore the homomorphism $\phi: S \to P_1 \times P_2$ given by $a \mapsto (\phi_1(a), \phi_2(a))$ separates $s$ from $Y_0$.
	
	As $N$ is residually finite, there exists a finite semigroup $Q$ and homomorphism $\psi:N \to Q$ such that $\psi(0) \notin \psi(Z \setminus\{0\})$.
	Then the homomorphism $\phi \times \psi : S \times N \to (P_1 \times P_2) \times Q$ given by $(a, b) \mapsto (\phi(a), \psi(b))$ separates $(s, n)$ from $U$.
	
 (ii) As $S$ is WSS it is residually finite by \cite[Proposition 2.1]{me} and so there exists a finite semigroup $P$ and homomorphism $\phi: S \to P$ such that $\phi(s) \notin \phi(Y_i)$. 
 As $N$ is residually finite, there exists a finite semigroup $Q$ and homomorphism $\psi:N \to Q$ such that $\psi(z_i) \notin \psi(Z \setminus \{z_i\})$.
 Then $\phi \times \psi : S \times N \to P \times Q$ given by $(a, b) \mapsto (\phi(a), \psi(b))$ separates $(s, n)$ from $U$.
 This completes the proof that $S \times N$ is WSS\@.
\end{proof}

A characterisation for when a finite semigroup is weakly subsemigroup separability preserving is not known.
We leave this as one of the open problems concerning weak subsemigroup separability and direct products.
But before we state the open problems, we show that the direct product of two WSS semigroups is MSS\@.

\begin{thm}
\thlabel{2WSS}
The direct product of two weakly subsemigroup separable semigroups is monogenic subsemigroup separable.
\end{thm}

\begin{proof}
Let $S$ and $T$ be WSS semigroups.
Let $U =\langle (s, t) \rangle \leq S \times T$ be a monogenic subsemigroup and let $(x, y) \in (S \times T) \setminus U$.
If $ x \notin \pi_S(U)=\langle s \rangle$ then we can separate $(x, y)$ from $U$ by factoring through $S$ and using the weak subsemigroup separability of $S$.
By a similar argument, we can separate $(x, y)$ from $U$ if $y \notin \pi_T(U)$.

Now consider the case that $x \in \pi_S(U)$ and $y \in \pi_S(T)$.
Then $x=s^i$ and $y=t^j$, where $i, j \in \mathbb{N}$ such that $i\neq j$, $s^i \neq s^j$ and $t^i \neq t^j$.
Firstly we consider the case when at least one of $\langle s \rangle$ and $\langle t \rangle$ is infinite.
Without loss of generality, assume that $\langle s \rangle \cong \mathbb{N}$.
In this instance $s ^i \notin \langle s^{i+1}, s^{i+2}, \dots, s^{2i+1} \rangle=\{s^k \mid k \geq i+1\}$.
Hence we can separate $(x, y)$ from $\{(s^k, t^k) \mid k \geq i+1\}$ by factoring through $S$ and using the weak subsemigroup separability of $S$.
That is, there exists a finite semigroup $P_1$ and a homomorphism $\phi_1 : S \times T \to P_1$ such that $\phi_1(x,y) \neq \phi_1(s^n, t^n)$ for all $n \geq i+1$. 
As $S$ and $T$ are both WSS, they are both residually finite by \cite[Proposition 2.1]{me}.
Hence $S \times T$ is residually finite by \cite[Theorem 2]{Gray2009}.
Then we can separate $(x, y)$ from the finite set $\{(s^k, t^k) \mid k \leq i\}$.
That is, there exists a finite semigroup $P_2$ and homomorphism $\phi_2: S \times T \to P_2$ such that $\phi_2(x, y) \neq \phi_2(s^k, t^k)$ for $1 \leq k \leq i$.
Then the homomorphism $\phi: S \times T \to P_1 \times P_2$ given by $\phi(a, b)=(\phi_1(a, b), \phi_2(a, b))$ separates $(x, y)$ from $U$. 

The final case to consider is when both $\langle s \rangle$ and $\langle t \rangle$ are finite, in which case $U$ is finite.
As $S \times T$ is residually finite we can separate $(x, y)$ from $U$.
Hence $S \times T$ is MSS as desired.
\end{proof}

We conclude this section with some open problems.

\begin{op}
Is there a characterisation of when a finite semigroup is weakly subsemigroup preserving?	
\end{op}

\begin{op}
Is there a characterisation of when a direct product of two weakly subsemigroup separable semigroups is itself weakly subsemigroup separable?
\end{op}

\begin{op}
If the direct product of two semigroups is weakly subsemigroup separable, is at least one of the factors weakly subsemigroup separable?	
\end{op}

\section{Monogenic subsemigroup Separability}

We have already seen that if a direct product of two semigroups is MSS, then it need not be that both factors are MSS (\thref{N cross Z}).
Here we also show that the direct product of two MSS semigroups need not be MSS, even if one of the factors is finite (\thref{monogenic eg}).
In the process of constructing this example, we develop the necessary theory to determine when a finite semigroup is MSS-preserving (\thref{monogenic finite}).
We complete this section by showing the direct product of a WSS semigroup with a residually finite periodic semigroup is always MSS (\thref{final}).

We shall first turn our attention to finding an example of a direct product of two MSS semigroups which is not MSS\@.
Key to this is a particular group embeddable semigroup $A$, which we will define shortly.
In showing that $A$ is group embeddable we use a criterion given by Adian, which we now review.

\begin{definition}
Let $\langle \, X \mid R \, \rangle$ be a presentation.
Consider a relation $(u, w) \in R$.
The \emph{left pair} of $(u, w)$ is the pair $(x, y)$, where $x$ is the leftmost letter of $u$ and $y$ is the leftmost letter of $w$.
The \emph{right pair} of $(u, w)$ is the pair $(z, t)$, where $z$ is the rightmost letter of $u$ and $t$ is the rightmost letter of $w$.
The \emph{left graph of the presentation} is the graph with vertex set $X$ such that $\{x, y\}$ is an edge if and only if $(x, y)$ is a left pair of some relation.
Note that multiple edges and loops are allowed.
The \emph{right graph of the presentation} is the graph with vertex set $X$ such that $\{z, t\}$ is an edge if and only if $(z, t)$ is a right pair of some relation.
The presentation $\langle \, X \mid R \, \rangle$ is said to have no \emph{cycles} if both its left graph and right graph have no cycles (here loops are cycles and multiple edges create cycles).
\end{definition}

\begin{thm}
\thlabel{Adian}
\normalfont{(\cite[Theorem 2.3]{Adian})} \textit{If a presentation has no cycles then the natural mapping from the semigroup given by the presentation to the group given by the presentation is an embedding.}
\end{thm}

\begin{cor}
\thlabel{embedding}
The mapping $\phi$ given by
\[a \mapsto x, \quad b \mapsto y, \quad c \mapsto y^{-2}x^{-1}y,\]
from the semigroup $A$ given by the presentation $\langle \, a,b,c \mid ab^2c=b \, \rangle$ to the free group $\FG_2$ on the set $\{x, y\}$ is an embedding.
\end{cor}

\begin{proof}
By \thref{Adian}, the natural map from $A$ to the group $G$ given by the group presentation $\langle \, x,y,z \mid xy^2z=y \, \rangle$ is an embedding.
It can be easily seen that this group is free, by use of a single Tietze transformation eliminating the generator $z$.
\end{proof}

\begin{definition}
In the semigroup $A$ with presentation $\langle \, a,b,c \mid ab^2c=b \, \rangle$, the strings $abbc$ and $b$ represent the same element.
Therefore we can define a rewriting system on $A$ that replaces the string $abbc$ by $b$.
As the single rewriting rule is length reducing, the rewriting system is terminating.
Also, as $ab^2c$ does not overlap with itself, the rewriting system is locally confluent and hence confluent.
For more on rewriting systems see \cite[Section 1.1]{Book}.
Therefore, each element of $A$ is represented by a unique word in $\{a,b,c\}^+$, where this representative does not contain $abbc$ as a subword.  
Such a word is said to be in \emph{normal form}.
Note that a contiguous subword of a normal form word is also a word in normal form.
We shall therefore consider the underlying set of $A$ to be the set of all words over $\{a,b,c\}$ in normal form.
Multiplication is concatenation, except in the case where concatenation creates instances of $abbc$ as subwords, in which case the rewriting rule is applied to convert the product into normal form.
\end{definition}

\begin{lem}
\thlabel{positive}
Let $A$, $\FG_2$ and $\phi: A \to \FG_2$ be as in \thref{embedding}.
Then $\phi(w) \neq \epsilon$ for all $w \in A$.
\end{lem}

\begin{proof}
Let $w$ be an element of $A$.
We proceed by a case analysis based upon the number of occurrences of contiguous strings of the letter $c$ appearing in the word $w$.

\textbf{Case 1.} The first case in when there are no occurrences of the letter $c$ in $w$.
Then $w \in \{a, b \}^+$.
As $\phi$ rewrites an occurrence of $a$ with an $x$ and an occurrence of $b$ with a $y$, we have that $\phi(w)$ is a non-empty word, as desired.
 	
\textbf{Case 2.} Now we consider the case when $w$ contains precisely one string of the letter $c$.
We split into three subcases.

\textbf{Case 2a.} Consider $w=c^n$, where $n \geq 1$.
Observe that
\[\phi(c^n)=y^{-2}(x^{-1}y^{-1})^{n-1}x^{-1}y.\] 
Hence the assertion of the lemma holds.
For future cases note that $\phi(c^n)$ ends with the suffix $x^{-1}y$.

\textbf{Case 2b.} Consider $w=uc^n$, where $u \in \{a,b\}^+$ and $n \geq 1$.
We claim that when concatenating the words $\phi(u)$ and $\phi(c^n)$, there are at most two cancelling pairs of letters.
Indeed, if there were three cancelling pairs of letters, then $\phi(u)$ would end with $xy^2$. This follows since by Case (2a) we know that $\phi(c^n)$ begins with $y^{-2}x^{-1}$. 
Hence $u$ would end with $ab^2$.
But in this case $uc^n$ contains the string $ab^2c$, contradicting that $uc^n$ is in normal form.
It follows that $(x^{-1}y^{-1})^{n-1}x^{-1}y$ is a suffix of $\phi(uc^n)$ and hence $\phi(uc^n)$ is non-empty. 
Again we note that $\phi(uc^n)$ ends with $x^{-1}y$.
For future cases we consider when $\phi(uc^n)$ can begin with a negative letter.
In such a case, the entirety of $\phi(u)$ has been cancelled by $\phi(c^n)$.
By our analysis of cancellation between $\phi(u)$ and $\phi(c)$ there are only two options: $u=b$ or $u=b^2$.

\textbf{Case 2c.} Consider $w=uc^nv$, where $u \in \{a, b\}^*$, $v \in \{a, b\}^+$ and $n \geq 1$.
Then $\phi(v)$ consists of positive letters.
By Cases (2a) and (2b), $\phi(uc^n)$ ends with a positive letter.
Then when concatenating $\phi(uc^n)$ and $\phi(v)$, there can be no pairs of cancelling letters.
Hence $\phi(uc^nv)$ is non-empty as desired.

\textbf{Case 3.} Finally we consider the case when $w$ contains more than one string of the letter $c$.
We can decompose $w=w_1w_2\dots w_{k}v$ where:
\begin{itemize}[nolistsep, noitemsep]
	\item $w_1=uc^{n_1}$, where $u \in \{a,b\}^*$ and $n_1 \geq 1$;
	\item $w_i=u_ic^{n_i}$, where $u_i \in \{a,b\}^+$ and $n_i \geq 1$ for $2 \leq i \leq k$; 
	\item $v \in \{a,b\}^*$; and $k \geq 2$; 
\end{itemize}
By our previous case analysis, for each $i$ we have that $\phi(w_i)$ is a non-empty word.
We now claim that for $1 \leq i \leq k$, when we concatenate $\phi(w_i)$ and $\phi(w_{i+1})$ there is at most one cancelling pair of letters.
We have already observed in Cases (2a) and (2b) that $\phi(w_i)$ must end with $x^{-1}y$.
Therefore for cancellation to occur, $\phi(w_{i+1})$ must begin with $y^{-1}$.
By the final observation of Case (2b), $\phi(w_{i+1})$ can only begin with a negative letter if $u_{i+1}=b$ or $u_{i+1}=b^2$.
In the first case, $\phi(w_{i+1})=y^{-1}(x^{-1}y^{-1})^{n_{i+1}-1}x^{-1}y$ and there is precisely one cancelling pair when we concatenate $\phi(w_i)$ and $\phi(w_{i+1})$.
In the second case, $\phi(w_{i+1})=(x^{-1}y^{-1})^{n_{i+1}-1}x^{-1}y$ and no cancellation occurs when we concatenate $\phi(w_i)$ and $\phi(w_{i+1})$ completing the proof of the claim.
Note that if cancellation occurs, than the cancelling pair is $yy^{-1}$.

Now consider $\phi(w_1)\phi(w_2)\dots \phi(w_k) \phi(v)$.
As already observed, $\phi(w_1)$ and $\phi(w_2)$ both end with $x^{-1}y$.
By the claim of the previous paragraph, there is at most one cancelling pair of letters, $yy^{-1}$, between $\phi(w_1)$ and $\phi(w_2)$. 
So it must be the case that $\phi(w_1w_2)$ also ends with $x^{-1}y$.
Continuing in this manner, we conclude that $\phi(w_1w_2 \dots w_k)$ ends with $x^{-1}y$.
As $\phi(v)$ is either empty or consists of positive letters, there is no cancellation between $\phi(w_1w_2 \dots w_k)$ and $\phi(v)$ and we conclude $\phi(w)$ is non-empty, completing the proof of this case and of the lemma.
\end{proof}

Before we proceed to show that $A$ is MSS but not WSS, we introduce the notion of stability.
This is needed in showing that $A$ is not WSS\@.

\begin{definition}
A semigroup $S$ is stable if both 
\[a \, \mathcal{J} \, ab \implies a \, \mathcal{R} \, ab  \quad \text{and} \quad a \, \mathcal{J} \, ba \implies a \, \mathcal{L} \, ba.\]
\end{definition}

\begin{lem}
	\thlabel{stable}
	\normalfont{\cite[Theorem A.2.4]{Rhodes}} \textit{Finite semigroups are stable.} 
\end{lem}

\begin{rem}
\thlabel{rem:stable}
In a stable semigroup if $x \, \mathcal{J} \, x^2$, then we have that $x \, \mathcal{R} \, x^2$ and $x \, \mathcal{L} \, x^2$.
In other words $x \, \mathcal{H} \, x^2$.
Then $H \cap H^2 \neq \emptyset$, where $H$ is the $\mathcal{H}$-class of $x$, and we conclude that $H$ is a group.
\end{rem}

We are now ready to establish the separability properties of the semigroup $A$.

\begin{prop}
	\thlabel{monogenic counter}
	The semigroup $A$ given by the presentation $\langle \, a, b, c \mid ab^2c=b \, \rangle$ is monogenic subsemigroup separable but not weakly subsemigroup separable.
\end{prop}

\begin{proof}
	Let $T=\langle u \rangle \leq A$ be a monogenic subsemigroup and let $v \in A \! \setminus \! T$.
	Let $\FG_2$ and $\phi: A \to \FG_2$ be as in \thref{embedding}. 
	As $\phi$ is an embedding, we have that $\phi(v) \notin \phi(T)$. 
	Let $H \leq \FG_2$ be the cyclic subgroup with generator $\phi(u)$. 
	First we show that $\phi(v) \notin H$. 
	
	For a contradiction, assume that $\phi(v) \in H$. 
	As $\phi(v) \notin \phi(T)$, we have that $\phi(v) \notin \{\phi(u)^n \mid n \in \mathbb{N}\}$. 
	By \thref{positive}, $\phi(v) \neq \epsilon$ and therefore $\phi(v) \neq \phi(u)^0$. 
	Therefore $\phi(v)=\phi(u)^{-n}$ for some $n \in \mathbb{N}$.
	But then $\phi(vu^n)=\phi(v)\phi(u^n)=\epsilon$.
	This contradicts \thref{positive} and we conclude $\phi(v) \notin H$.

	As $\FG_2$ is weakly subgroup separable \cite[Theorem 5.1]{Hall49}, there exists a finite group $G$ and homomorphism $\psi: \FG_2 \to G$ such that $\psi(\phi(v)) \notin \psi(H)$. 
	In particular, $\psi \circ \phi :A \to G$ is a homomorphism from $A$ to a finite semigroup such that $(\psi \circ \phi)(u) \notin (\psi \circ \phi)(T)$. 
	Hence $A$ is MSS\@. 

	To show that $A=\langle \, a, b, c \mid ab^2c=b \, \rangle$ is not weakly subsemigroup separable, consider the subsemigroup $V=\langle b^2, b^3 \rangle$. 
	As $\langle b \rangle \cong \mathbb{N}$, we have $b \in A \setminus V$. 
	Let $P$ be a finite semigroup and let $\sigma: A \to P$ be a homomorphism. 
	The relation $ab^2c=b$ ensures that $b \mathcal{J} b^2$ and hence $\sigma(b) \mathcal{J} \sigma(b^2)$. 
	But as $P$ is finite we have that the $\mathcal{H}$-class of $\sigma(b)$ is a group by \thref{stable} and \thref{rem:stable}.
	Hence $\sigma(V)=\langle \sigma(b) \rangle$ is a finite cyclic group and in particular $\sigma(b) \in \sigma(V)$. 
	Therefore $A$ is not weakly subsemigroup separable. 
\end{proof}

The semigroup $A$ will prove crucial in characterising which finite semigroups are MSS-preserving.
In fact our characterisation will stretch to residually finite periodic semigroups.
Note that a periodic semigroup is MSS if and only if it is residually finite. 
Periodic residually finite semigroups are MSS because every monogenic subsemigroup of such a semigroup is finite.
The fact that MSS semigroups are residually finite is a consequence of \cite[Proposition 2.1]{me}.
Although the statement of \cite[Proposition 2.1]{me} asserts that WSS semigroups are residually finite, the proof of this assertion only uses weak subsemigroup separability to separate elements from monogenic subsemigroups, and thus also serves as a proof that MSS semigroups are residually finite.

\begin{thm}
\thlabel{monogenic finite}
	A residually finite periodic semigroup is monogenic sub\-semigroup separability preserving if and only if it is a union of groups.	
\end{thm}

\begin{proof}
$(\Leftarrow)$ Let $T$ be a residually finite periodic semigroup which is a union of groups.
Let $S$ be a MSS semigroup and let $U=\langle (s, t) \rangle \leq S \times T$. 
Let $(x,y) \in (S \times T) \! \setminus \! U$. 
We separate into cases.
Note some cases may overlap.

\textbf{Case 1.} Suppose $\langle s \rangle \leq S$ is finite.
As $T$ is periodic, we also have that $\langle t \rangle$ is finite.
Therefore $U$ is also finite.
As $S$ is MSS it is also residually finite (\cite[Proposition 2.1]{me}) and therefore $S \times T$ is residually finite.
Hence we can separate $(x, y)$ from $U$.

\textbf{Case 2.} Suppose that $x \notin \pi_S(U)$; that is, $x \notin \langle s \rangle \leq S$. 
We can separate $(x, y)$ from $U$ by factoring through $S$ and invoking the monogenic subsemigroup separability of $S$. 

\textbf{Case 3.} Suppose that $y \notin \pi_T(U)$. 
As $\pi_T(U)=\langle t \rangle$ and $T$ is periodic, we have that $\pi_T(U)$ is finite.
So we can separate $(x, y)$ from $U$ by factoring though $T$ and invoking the residual finiteness of $T$.

\textbf{Case 4.} Now suppose that $\langle s \rangle \cong \mathbb{N}$, $x \in \pi_S(U)$ and $y \in \pi_T(U)$. 
Let $r \in \mathbb{N}$ be minimal such that $t^{r+1}=t$.
Such an $r$ exists as $T$ is periodic and a union of groups.
As $x \in \pi_S(U)$, we have that $x=s^i$ for some $i \in \mathbb{N}$.
As $y \in \pi_T(U)$, we have that $y=t^j$ for some $j \in \{1, 2, \dots, r\}$.
Observe that
\[\pi_S(U \cap (S \times \{t^j\}))=\{s^{j+rn} \mid n \geq 0\}.\]
As $(x, y) \notin U$, it must be the case that $i \not\equiv j \pmod{r}$.
We split into two subcases.

\textbf{Case 4a.} First we will deal with the case $j=r$. 
In this case, $y$ is an idempotent and we have 
\[\pi_S(U \cap(S \times \{t^r\}))=\langle s^r \rangle.\]
Hence $x=s^i \notin \langle s^r \rangle$. 
As $S$ is MSS, there exists a finite semigroup $P_1$ and homomorphism $\phi_1 : S \to P_1$ such that $\phi(x) \notin \phi(\langle s^r \rangle)$. 
As $T$ is residually finite, there exists a finite semigroup $P_2$ and homomorphism $\phi_2 : T \to P_2$ such that $\phi_2(y) \neq \phi(\pi_T(U) \setminus \{y\})$.
Then $\phi: S \times T \to P_1  \times P_2$ given by $(a, b) \mapsto (\phi_1(a), \phi_2(b))$ is a homomorphism that separates $(x, y)$ from $U$.

\textbf{Case 4b.} Now assume that $j \in \{1, 2, \dots, r-1\}$.
Let $k$ be such that $j+k=r$.
Then as $i \not \equiv j \pmod{r}$, we have that  $ i+k \not \equiv j+k \pmod{r}$.
Hence $(s^{i+k}, t^{j+k})=(s^{i+k}, t^r) \notin U$.
By Case 4a we can separate $(s^{i+k}, t^r)$ from $U$.

We now show that we can separate $(x, y)=(s^i, t^j)$ from $U$.
For a contradiction suppose it cannot be separated.
Let $\sim$ be a finite index congruence which separates $(s^{i+k}, t^r)$ from $U$.
As $(s^i, t^j)$ cannot be separated from $U$, there exists $\ell \in \mathbb{N}$ such that $(s^i, t^j)\sim (s, t)^{\ell}$.
But then
\[(s^{i+k}, t^r)=(s^i, t^j)(s, t)^k \sim (s, t) ^{\ell} (s, t)^k \in U.\]
This contradicts that $\sim$ separates $(s^{i+k}, t^r)$ from $U$.
Hence $(x, y)$ can be separated from $U$, completing the proof of this case and establishing that $S \times T$ is MSS\@.

$(\Rightarrow)$ Now suppose that $T$ is a residually finite periodic semigroup which is not a union of groups.
We will show that $T$ is not MSS-preserving.
As $T$ is not a union of groups, there exists an element $t \in T$ such that for all $n \geq 2$ we have that $t^n \neq t$.
Let $m$ be the index and let $r$ be the period of the monogenic subsemigroup $\langle t \rangle$.
We have that $m \geq 2$ and $t^m=t^{m+r}$.
Let $i=m-1$.
	
	As before, let $A$ be the monogenic subsemigroup separable semigroup given by the presentation $\langle \, a,b,c \mid ab^2c=b \, \rangle$. 
	It is easy to show that a semigroup $S$ is MSS if and only if $S^1$ is MSS.
	Therefore $A^1$ is also MSS.
	Let $U=\langle (b, t) \rangle \leq A^1 \times T$. 
	Observe that 
	\begin{equation}
	\label{10}
	\pi_{A^1}(U \cap (A^1 \times \{t^{i+r}\})=\{b^{i+nr} \mid n \geq 1\}.
	\end{equation}
	Then $(b^i, t^{i+r}) \notin U$. 
	
	Let $d \in \{m, m+1, \dots, m+r-1\}$ be such that $d \equiv 0 \pmod{r}$.
	Then it must be the case that $t^d$ is an idempotent.
	Furthermore, $t^{i+r}t^{d}=t^{i+r}$.
	Let $\rho$ be a finite index congruence on $A^1 \times T$. 
	Then
	\[[(a, t^d)]_\rho[(b^2, t^d)]_\rho[(c, t^d)]_\rho=[(ab^2c,t^d)]_\rho=[(b, t^d)]_\rho.\]
	As also $[(b, t^d)]_\rho[(b, t^d)]_\rho=[(b^2, t^d)]_\rho$ we have that $[(b,t^d)]_\rho \mathcal{J} [(b,t^d)]_\rho^2$. 
	Since $(A^1 \times T)/\rho$ is finite, we conclude that the $\mathcal{H}$-class of $[(b, t^d)]_{\rho}$ is a finite group by \thref{stable} and \thref{rem:stable}.
	In particular there exists $k >2$ such that $[(b, t^d)]_\rho^k=[(b, t^d)]_\rho$.
	From this we conclude that $[(b, t^d)]_\rho=[(b^{1+(k-1)n}, t^d)]_\rho$ for all $n \in \mathbb{N}$.
	Adopting the convention that if $i=1$ then $b^{i-1}$ is the identity of $A^1$, we observe that 
	\begin{align*}
	[(b^i, t^{i+r})]_{\rho}&=[(b, t^d)]_\rho[(b^{i-1}, t^{i+r})]_{\rho}\\&=[(b^{1+(k-1)r}, t^d)]_\rho[(b^{i-1}, t^{i+r})]_{\rho}\\&=[(b^{i+(k-1)r}, t^{i+r})]_\rho. \end{align*}
	From Equation (\ref{10}), we have that $(b^{i+(k-1)r}, t^{i+r}) \in U$.
	So $\rho$ does not separate $(b^i, t^{i+r})$ from $U$.
	As $\rho$ was arbitrary we conclude that $A^1 \times T$ is not MSS and in particular $T$ is not MSS-preserving.
\end{proof}

\begin{rem}
From \thref{finite strong} and \thref{monogenic finite}, we have that the class of finite MSS-preserving semigroups is contained within the class of finite SSS-preserving semigroups. 
However, as already observed, there is no such containment between the class of finite WSS-preserving semigroups with either finite SSS-preserving or finite MSS-preserving semigroups.
\end{rem}

The following example is a consequence of the proof of \thref{monogenic finite}.
It shows that the direct product of two MSS semigroups need not be MSS, even when one of the factors is finite.

\begin{eg}
\thlabel{monogenic eg}
The direct product of $A=\langle \, a, b, c \mid ab^2c=b \, \rangle$ and the two element zero semigroup $N=\{x, 0\}$ (with the zero element 0) is not monogenic subsemigroup separable.
This follows from the proof of \thref{monogenic finite} and as $N$ is not a union of groups.
\end{eg} 

Although it is not true that the direct product of an MSS semigroup with a residually finite semigroup is necessarily MSS, if we strengthen the assumption of monogenic subsemigroup separability to weak subsemigroup separability, we obtain a positive result, which we present below.
\thref{final} can be seen as a successor of \thref{monogenic finite}, be it can also be viewed a variation of \thref{2WSS}.
The similarities between \thref{final} and \thref{2WSS} are discussed in \thref{rem}.

\begin{prop}
\thlabel{final}
	The direct product of a weakly subsemigroup separable semigroup $S$ with a residually finite periodic semigroup $T$ is monogenic subsemigroup separable. 
\end{prop}

\begin{proof}
Let $U =\langle (s, t) \rangle \leq S \times T$ and let $(x, y) \in (S \times T) \setminus U$. 
We split into cases.

\textbf{Case 1.} Suppose $\langle s \rangle \leq S$ is finite.
As $T$ is periodic, we also have that $\langle t \rangle$ is finite.
Therefore $U$ is also finite.
As $S$ is MSS it is also residually finite (\cite[Proposition 2.1]{me}) and therefore $S \times T$ is residually finite.
Hence we can separate $(x, y)$ from $U$.

\textbf{Case 2.} Suppose that $x \notin \pi_S(U)$. 
Then $x \notin \langle s \rangle \leq S$. 
So we can separate $(x, y)$ from $U$ by factoring through $S$ and invoking the weak subsemigroup separability of $S$. 

\textbf{Case 3.} Now suppose that $\langle s \rangle \cong \mathbb{N}$ and that $x \in \pi_S(U)$.
Then $x=s^i$ for some $i \in \mathbb{N}$.
As $S$ is WSS, we can separate $s^i$ from $\langle s^{i+1}, s^{i+2}, \dots, s^{2i-1} \rangle =\{s^k \mid k \geq i+1\}$.
Hence we can separate $(x, y)$ from $\{(s^k, t^k) \mid k \geq i+1\}$ by factoring through $S$ and using the weak subsemigroup separability of $S$.
That is, there exists a finite semigroup $P_1$ and a homomorphism $\phi_1 : S \times T \to P_1$ such that $\phi_1(x,y) \neq \phi_1(s^n, t^n)$ for all $n \geq i+1$. 
As $S$ is WSS, it is residually finite by \cite[Proposition 2.1]{me}.
Hence $S \times T$ is residually finite by \cite[Theorem 2]{Gray2009}.
Then we can separate $(x, y)$ from the finite set $\{(s^k, t^k) \mid k \leq i\}$.
That is, there exists a finite semigroup $P_2$ and homomorphism $\phi_2: S \times T \to P_2$ such that $\phi_2(x, y) \neq \phi_2(s^k, t^k)$ for $1 \leq k \leq i$.
Then the homomorphism $\phi: S \times T \to P_1 \times P_2$ given by $\phi(a, b)=(\phi_1(a, b), \phi_2(a, b))$ separates $(x, y)$ from $U$. 
\end{proof}

\begin{rem}
\thlabel{rem}
If every residually finite periodic semigroup is WSS, then \thref{final} becomes a specific instance of \thref{2WSS}.
However, it is not known whether every residually finite periodic semigroup is WSS\@. 
This is left as an open question at the end of this section.
\end{rem}

We conclude this section, and the paper, with some open problems.

\begin{op}
Is there a characterisation of when the direct product of two monogenic subsemigroup separable semigroups is itself monogenic subsemigroup separable?
\end{op}

\begin{op}
If the direct product of two semigroups is monogenic subsemigroup separable, is at least one of the factors monogenic subsemigroup separable?
\end{op}

\begin{op}
Is every residually finite periodic semigroup weakly subsemigroup separable?
\end{op}


\begin{thebibliography}{99}
	\bibliographystyle{plain}
	\bibitem{Adian} S. I. Adian. Defining relations and algorithmic problems for groups and semigroups. \textit{Proc.\ Steklov Inst.\ Math.}, 85:1--152, 1966.
	\bibitem{Allenby} R. B. J. T. Allenby and R. J. Gregorac. On locally extended residually finite groups. \textit{Conference on Group Theory}. Springer Berlin Heidelberg. 9--17, 1973
	\bibitem{Book} R. V. Book and F. Otto, \textit{String-rewriting systems}, Texts and Monographs in Computer Science, Springer-Verlag, New York, 1993.
	\bibitem{Golubov70} \`E. A. Golubov. Finite separability in semigroups. 
	\textit{Sibirsk.\ Mat.\ \v{Z}.}, 11:1247--1263, 1970.
	\bibitem{Golubov72} \`E. A. Golubov. The direct product of finitely divisible semigroups. (Russian)
	\textit{Ural.\ Gos.\ Univ.\ Mat.\ Zap.},
	8:28--34, 1972
	\bibitem{Gray2009} R. Gray and N. Ru\v{s}kuc. On residual finiteness of direct products of algebraic systems. \textit{Monatsh.\ Math.}, 158:63--69, 2009.
	\bibitem{Hall49} M. Hall. Coset representation in free groups. \textit{Trans. Amer. Math. Soc.}, 67:421--432, 1949.
	\bibitem{Howie95} J. M. Howie. \textit{Fundamentals of Semigroup Theory}. LMS monographs. Clarendon Press, 1995.
	\bibitem{Malcev} A.I. Mal'cev. On homomorphisms onto finite groups. \textit{Ivanov. Gos. Ped. Inst. U\v{c}en. Zap.}, 18:49--60, 1959. 
	\bibitem{me} C. Miller, G. O'Reilly, M. Quick and N. Ru\v{s}kuc. On separability finiteness conditions in semigroups. arXiv2006.08499v1, 2020. 
	\bibitem{Rhodes} J. Rhodes and B. Steinberg. \textit{The q-theory of Finite Semigroups}. Springer Sciences \& Business Media, 2009.
	\bibitem{Stebe} P. Stebe. Residual finiteness of a class of knot groups. \textit{Comm.\ Pure Appl.\ Math.}, 21:563--583, 1968
	\bibitem{de Witt} B. de Witt. Residual finiteness and related properties in monounary
	algebras and their direct products. \textit{Algebra Univers.} 82, 32, 2021. https://doi.org/10.1007/s00012-021-00727-4
\end{thebibliography}
\end{document}